\documentclass{amsart}
\usepackage{amssymb}
\usepackage{amsmath,amsthm,amscd,amssymb, mathrsfs}

\usepackage{latexsym}
\usepackage{color}

\numberwithin{equation}{section}

\theoremstyle{plain}
\newtheorem{theorem}{Theorem}[section]
\newtheorem{lemma}[theorem]{Lemma}
\newtheorem{corollary}[theorem]{Corollary}
\newtheorem{proposition}[theorem]{Proposition}

\theoremstyle{definition}

\theoremstyle{remark}
\newtheorem{remark}[theorem]{Remark}

\newtheorem{case[theorem]}{Case}

\newcommand{\nothing}[1]{{}}

\title[Cardinalities of distance sets]{Distance sets of two subsets of   vector spaces over  finite fields}
\author{ Doowon Koh and Hae-Sang Sun}
\date{\today}

\address{Department of Mathematics\\
Chungbuk National University \\
Cheongju, Chungbuk 361-763 Korea} \email{koh131@chungbuk.ac.kr}
\address{Department of Mathematics\\
Chungbuk National University \\
Cheongju, Chungbuk 361-763 Korea} \email{haesang@chungbuk.ac.kr}
\thanks{
This research was supported by Basic Science Research Program through the National Research Foundation
of Korea(NRF) funded by the Ministry of Education, Science and Technology(2012010487, 20100023248) }


\begin{document}

\begin{abstract}
We investigate the size of the distance set determined by two subsets of  finite dimensional vector
spaces over  finite fields. A lower bound of the size is given explicitly in terms of cardinalities of
the two subsets. As a result, we improve upon the results by Rainer Dietmann \cite{Di12}. In the case
that one of the subsets is a product set, we obtain further improvement on the estimate.
\end{abstract}
\maketitle

\section{Introduction}
Let $E, F$ be finite subsets of $\mathbb R^d.$ The distance set determined by $E$ and $F$ is defined
by
$$\Delta(E, F) =\{ |x-y|: x\in E, y\in F\},$$
where $|\cdot|$ denotes the standard norm on $\mathbb R^d.$ Several attentions have been paid to
estimating the cardinality of distance set $\Delta(E, F).$ In the case of $E=F,$   Erd\H{o}s
\cite{Er46} first addressed this problem and  showed that
$$ |\Delta(E,E)|\gg|E|^{\frac{1}{d}}$$
where $|E|$ denotes the number of elements in $E$. Here we use $\gg$ conventionally. Tha is to say, there exists a $c>0$ such that $\Delta(E,E)>c |E|^{\frac{1}{d}}$ for all $E$. Taking the set $E$ as a piece of the integer lattice, the Erd\H{o}s distance conjecture
says that for every $\varepsilon>0,$ there exists a $c_{\varepsilon}>0$ such that
$$ |\Delta(E,E)|\geq c_{\varepsilon} |E|^{\frac{2}{d}-\varepsilon}.$$
In dimension two, the conjecture has recently been solved by Guth and Katz \cite{GK10}, who proved that
$$  |\Delta(E,E)|\gg \frac{|E|}{\log|E|}.$$
However, the Erd\H{o}s distance conjecture is still open for higher dimensions. See \cite{KT04},
\cite{SV04}, \cite{SV05}, and the references contained therein for recent developments on the
Erd\H{o}s distance problem in higher dimensions.

As an analog of the Euclidean Erd\H{o}s distance problem, Bourgain, Katz, and Tao \cite{BKT04} posed
and studied the finite field version of the Erd\H{o}s distance problem in two dimensions. The
Erd\H{o}s distance problem in the finite field setting has been recently studied by various
researchers (see \cite{CEHIK09}, \cite{IK08},  \cite{KS}, \cite{KS12}, \cite{V08}, \cite{Vi08}, and
\cite{Vu08}). Let $\mathbb F_q^d$ be a $d$-dimensional vector space over the finite field $\mathbb
F_q$ with $q$ elements. We shall always assume that the characteristic of $\mathbb F_q$ is greater
than two. In the finite field setting, given two sets $E, F\subset \mathbb F_q^d,$ the distance set is
defined similarly by
$$ \Delta(E, F)=\{\|x-y\|\in \mathbb F_q: x\in E, y\in F\},$$
where $\|\cdot\|$ is defined by $\|m\|=m_1^2+\dots+m_d^2$ for $m=(m_1,\dots, m_d)\in \mathbb F_q^d.$
Here, observe that the function $\|\cdot \|$ on $\mathbb F_q^d$ is not a norm, but the value is invariant under the rotations in $\mathbb F_q^d$.

Assuming that $E\subset \mathbb F_p^2$ with prime $p\equiv 3 (\mbox{mod}~4),$ the aforementioned
authors \cite{BKT04} proved that if $|E|\le p^{2-\varepsilon}$ for some $\varepsilon >0$, then there
exists a $\delta(\varepsilon)>0$ such that $|\Delta(E,E)|\gg
|E|^{\frac{1}{2}+\delta(\varepsilon)}.$ However, the value $\delta(\varepsilon)$ was not given in an
explicit form. Furthermore, this result can not be obtained for general finite fields, because one may take $E=\mathbb F_p\times \mathbb F_p$ for the prime field $\mathbb F_p$ of $\mathbb F_q.$

For general fields $\mathbb F_q,$  Iosevich and Rudnev  \cite{IR07} obtained results on lower bounds with explicit
exponents for the size of distance sets for $\mathbb F_q^d$, $d\geq 2$. More precisely, they proved
that if $E\subset \mathbb F_q^d$ and $ |E|\gg q^{\frac{d}{2}}$,
then
$$ |\Delta(E, E)|\gg\min\left\{ q, \frac{|E|}{q^{\frac{d-1}{2}}}\right\}.$$

In \cite{Sh06}, Shparlinski derived an explicit lower bound of the number of the distances between
arbitrary two sets: If $E, F\subset \mathbb F_q^d,$ then
\begin{equation}\label{Sh}
|\Delta(E,F)| > \frac{|E||F|q}{q^{d+1}+|E||F|}\geq
\frac{1}{2} \min \left\{ q, \frac{|E||F|}{q^d}\right\}.
\end{equation}
Dietmann \cite{Di12} recently
obtained a new lower bound for $|\Delta(E, F)|.$ In fact, he proved that if $E, F\subset \mathbb
F_q^d, |F|\geq |E|,$ and $|E||F|\geq (900+\log{q}) q^d,$ then
\begin{equation}
\label{Di} |\Delta(E,F)|\gg \left\{
\begin{array}
{ll}  \min\left\{ q, \frac{|F|}{q^{\frac{d-1}{2}} \log{q}} \right\}\quad &\mbox{for}~~d\geq 2\\
 \min\left\{ q, \frac{|E|^{\frac{1}{2}}|F|}{q \log{q}} \right\}\quad &\mbox{for}~~d=2
 \end{array}\right..
\end{equation}
In order to obtain an estimate on an average of a product of two spherical sums, he made use of the pigeonhole principle, which is a main reason for presence of the $\log q$ factor in (\ref{Di}). One might make a naive speculation that explicit Fourier analysis instead of the pigeonhole principle could remove the $\log q$ factor.

Another point worth of noting is that both results of Shparlinski and Dietmann are nontrivial only if
\begin{align}\label{the:cond}
|E||F|> q^{d}.
\end{align}
Here notice that the condition (\ref{the:cond}) is optimal for even $d$. In fact, if $d\geq 2$ is even
and $i^2=-1$ for some $i\in\mathbb F_q$, then setting $E=F$ and
\begin{equation}\label{exa}
E=\left\{ (t_1, it_1, \dots, t_j, it_j, \dots, t_{d/2}, i t_{d/2}) \in \mathbb F_q^d:  t_j\in \mathbb
F_q, j=1,\dots, d/2 \right\},
\end{equation}
it can be easily shown that $|E||F|=q^{d}$ and $|\Delta(E,F)|=1$. On the other hand, absence of such
an example for odd $d$ leads us to speculate that (\ref{the:cond}) might be relaxed further.

There are three aims of present paper. The first two of them are responses to the previous
speculations, which are consequences of main theorems in Section \ref{main:thms}. First of all, we
shall observe that Dietmann's result can be improved so that the $\log{q}$ factor can be eliminated in
(\ref{Di}). In fact, in Corollary \ref{bett:dietm}, we have : If $|E||F|\gg q^d$ and $|E|\leq |F|$,
then we have
\begin{align*}
|\Delta(E,F)|\gg
\begin{cases}
\min\left\{q,\frac{|F|}{q^{\frac{d-1}{2}}}\right\}&\mbox{ if }d\geq2\\
\min\left\{q,\frac{|E|^{\frac{1}{2}}|F|}{q}\right\}&\mbox{ if }d=2.
\end{cases}
\end{align*}
Secondly, we show that in certain cases, a condition milder than (\ref{the:cond}) assures
non-triviality of distance sets. For example, in Theorem \ref{main1}, we show that if $d\geq 3$ is odd
and $1\leq |E|< q^{\frac{d-1}{2}}$, then
$$|\Delta(E,F)|\geq \min \left\{ \frac{q}{2}, \frac{|E||F|}{8q^{d-1}}\right\}$$ and, therefore,
$|\Delta(E,F)|>1$ if $|E||F|>8q^{d-1}$. Finally, we also show in Theorem \ref{main3} that if one of
the two subsets of $\mathbb F_q$ is a product set, then much stronger lower bound for the size of
distance set is obtained.

For precise statements and more explanations, please refer to the theorems and remarks in Section
\ref{main:thms}

\section{Discrete Fourier analysis}

Iosevich and Rudnev \cite{IR07} adapted the discrete Fourier analysis to measure the size of distance
sets in the finite field setting. As a result, they developed a powerful machinery for deriving
results on the Erd\H{o}s distance problem. In this section, we review it and collect estimates on
several quantities, namely the Fourier transform of spheres, counting function of points with a given
distance, and spherical sums, which are involved in the lower bound for the distance set.

We begin with the definition of the Fourier transform. Given a function $f:\mathbb F_q^d \to \mathbb
C$, the {\it Fourier transform} of $f$ is given by the form
$$ \widehat{f}(m)=\frac{1}{q^d}\sum_{x\in \mathbb F_q^d} \chi(-m\cdot x) f(x)
\quad \mbox{for}~~m\in \mathbb F_q^d.$$ Here, and throughout this paper, we denote by $\chi$ a fixed
nontrivial additive character of $\mathbb F_q$. It can be easily checked that the results on the
distance problems are independent of the choice of the character. Recall that the {\it orthogonality
relation} of $\chi$ says that
$$ \sum_{x\in \mathbb F_q^d} \chi(m\cdot x) =\left \{
\begin{array}
{ll} 0\quad&\mbox{if}~~m\neq (0,\dots,0)\\
     q^d \quad&\mbox{if}~~m= (0,\dots,0)\end{array}\right..$$
The following {\it Fourier inversion formula} follows immediately from a direct application of the
orthogonality relation of $\chi$:
$$ f(x)=\sum_{m\in \mathbb F_q^d} \chi(m\cdot x) \widehat{f}(m).$$
The discrete version of {\it Plancherel's theorem} says that:
$$ \sum_{m\in \mathbb F_q^d} |\widehat{f}(m)|^2 = q^{-d} \sum_{x\in \mathbb F_q^d} |f(x)|^2.$$
From now on, by abuse of notations, we identify the symbol $E$ for a subset $E\subset \mathbb F_q^d$
with the characteristic function $\chi_{E}$ on $E$. Then the Plancherel theorem for $E$ is interpreted as
$$ \sum_{m\in \mathbb F_q^d} |\widehat{E}(m)|^2 = q^{-d}\sum_{x\in \mathbb F_q^d} |E(x)|^2= q^{-d}|E|.$$

Let us denote by $G, K, S$ the {\it Gauss sum}, {\it Kloosterman sum}, and {\it Sali\'e sum},
respectively. In other words, for $a, b\in \mathbb F_q^*,$ let us set
$$ G=\sum_{s\in\mathbb F_q^*} \eta(s) \chi(s),\quad
K=\sum_{s\in\mathbb F_q^*}  \chi(as+bs^{-1}), \quad\mbox{and} \quad S=\sum_{s\in\mathbb F_q^*} \eta(s)
\chi(as+bs^{-1}),$$ where $\eta$ denotes the quadratic character of $\mathbb F_q^*:=\mathbb
F_q\setminus\{0\}.$ It is well known that they satisfy
\begin{equation}\label{expsum}|G|=\sqrt{q}, ~~|K|\leq 2 \sqrt{q}, ~~\mbox{and}~~ |S|\leq 2 \sqrt{q}.\end{equation}
For proofs of estimates on these exponential sums,  see \cite[p.193]{LN97} and
\cite[pp.322-323]{IK04}.

For each $t\in \mathbb F_q$, we define a sphere with radius $t$ as the set
$$ S_t=\{x\in \mathbb F_q^d: \|x\|=t\}.$$
The Fourier transform on $S_t$ is closely related to aforementioned exponential sums.
It was proved in \cite{IK10} that for $t\in \mathbb F_q, m\in \mathbb F_q^d,$
\begin{equation}\label{Sf}
\widehat{S_t}(m) = q^{-1} \delta_0(m) + q^{-d-1}\eta^d(-1) G^d \sum_{r \in {\mathbb F}_q^*}
\eta^d(r)\chi\Big(tr+ \frac{\|m\|}{4r}\Big),
\end{equation}
where $\delta_0(m)=1$ if $m=(0, \ldots, 0)$ and $\delta_0(m)=0$ otherwise. Given two sets $E, F\subset
\mathbb F_q^d,$  we consider a counting function $\nu: \mathbb F_q \to \mathbb N \cup \{0\}$ defined
by
$$\nu(t)=|\{(x,y)\in E\times F: \|x-y\|=t\}|\quad \mbox{for}~~t\in \mathbb F_q.$$
For $E\subset \mathbb F_q^d$, let us set
$${\mathfrak M}(E):= \max_{r\in \mathbb F_q} \sum_{m\in S_r} |\widehat{E}(m)|^2,\mbox{ and }
{\mathfrak M}^*(E):= \max_{r\in \mathbb F_q^*} \sum_{m\in S_r} |\widehat{E}(m)|^2.$$  Note that $\sum_{m\in S_r} |\widehat{E}(m)|^2$ is a finite analog of {\it spherical average} $\int_{S^{d-1}}|\widehat{\mu}(t\omega)|^2d\omega$ in the classical Falconer distance problem. Here $S^{d-1}$ is a $(d-1)$-dimensional sphere and $\mu$ is a Borel measure. Refer to \cite{IR07} for the details.

The three quantities $\widehat{S_t}$, $\nu(t)$, and ${\mathfrak M}(E)$ are closely related by the
following lemma, (\ref{L1}), and  (\ref{L2}).
\begin{lemma}
\begin{equation}\label{vt}
\nu(t)=q^{2d}\sum_{m\in \mathbb F_q^d} \widehat{S_t}(m) \overline{\widehat{E}}(m)
\widehat{F}(m).
\end{equation}
\end{lemma}
\begin{proof}
This can be checked by applying the Fourier inversion formula to $S_t(x-y)$. For each $t\in \mathbb
F_q,$ we have
$$ \nu(t)=\sum_{x\in E, y\in F} S_t(x-y) =\sum_{x,y\in \mathbb F_q^d} E(x)F(y)
\sum_{m\in \mathbb F_q^d} \chi(m\cdot (x-y)) \widehat{S_t}(m).$$ Then (\ref{vt}) follows from the
definition of the Fourier transform.
\end{proof}

Now in the following subsections we collect estimates and expressions for the quantities $\widehat{S_t}(m)$, $\nu(t)$, and
$\mathfrak M(E)$, which are necessary to prove main theorems in Section \ref{main:thms}.

\subsection{Fourier transform of $S_t$}
Clearly by definition we have
\begin{align}\label{s:0}
\widehat{S_t}(0,\dots,0)=q^{-d}|S_t|\leq 2 q^{-1}\end{align} for $t\in \mathbb F_q$. We also have

\begin{proposition}\label{decay}
(1) Let $m\neq (0,0,\cdots,0)$. We have
\begin{align*}
|\widehat{S_t}(m)|&\leq
\begin{cases}
~~q^{-\frac{d}{2}}&\mbox{ if }d: even,~ t=0,\mbox{ and }\|m\|=0\\
2 q^{-\frac{d+1}{2}}&\mbox{
otherwise}
\end{cases}
\end{align*}
In particular, $|\widehat{S_t}(m)|\leq 2q^{-\frac{d+1}{2}}$ for all $t\in\mathbb F_q$ if $d$ is odd,
and for all $t\neq 0$ if $d$ is even.
(2) For $m, m^\prime \in \mathbb F_q^d,$ we have
$$ \sum_{t\in \mathbb F_q} \widehat{S_t}(m) \overline{\widehat{S_t}}(m^\prime)
=q^{-1} \delta_0(m)\delta_0(m^\prime) +
q^{-d-1} \sum_{s\in\mathbb F_q^*} \chi(s(\|m\|-\|m^\prime\|)).$$
\end{proposition}

\begin{proof}
Observe that for any $d\geq 2$, we have
$$\left|\sum_{r \in {\mathbb F}_q^*}
\eta^d(r)\chi\Big(tr+ \frac{\|m\|}{4r}\Big)\right|\leq 2q^{\frac{1}{2}}$$ unless $d$ is even, $t=0$,
and $\|m\|=0$. The statement (1) is an immediate consequence from (\ref{expsum}) and (\ref{Sf}).

For $m\in \mathbb F_q^d, t\in \mathbb F_q,$ write $ \widehat{S_t}(m)=q^{-1} \delta_0(m)+ R_t(m),$
where $R_t(m)$ denotes the second term of the right-hand side in (\ref{Sf}). It follows that for $m,
m^\prime \in \mathbb F_q^d,$
\begin{align}
\sum_{t\in \mathbb F_q} \widehat{S_t}(m) \overline{\widehat{S_t}}(m^\prime) &= \sum_{t\in \mathbb F_q}
q^{-2} \delta_0(m)
\delta_0(m^\prime) + q^{-1}\delta_0(m) \sum_{t\in \mathbb F_q} \overline{R_t}(m^\prime)\label{lem1:1} \\
&\phantom{bbbbbb}+ q^{-1}\delta_0(m^\prime)\sum_{t\in \mathbb F_q} R_t(m) +\sum_{t\in \mathbb F_q}
R_t(m)\overline{R_t}(m^\prime).\notag
\end{align}
By the orthogonality relation for $\chi$, the sums in the second and third terms in (\ref{lem1:1})
vanish. Thus, (\ref{lem1:1}) is equal to
\begin{align*}
&q^{-1} \delta_0(m) \delta_0(m^\prime) + q^{-d-2} \sum_{s,s^\prime \in \mathbb F_q^*}
\eta^d(s)\overline{\eta}^d(s^\prime) \chi\left(\frac{\|m\|}{4s}-\frac{\|m^\prime\|}{4s^\prime}\right)
\sum_{t\in \mathbb F_q}\chi((s-s^\prime)t) \\
&=q^{-1} \delta_0(m) \delta_0(m^\prime) + q^{-d-1} \sum_{s\in \mathbb F_q^*} \chi \left(
\frac{\|m\|-\|m^\prime\|}{4s}\right).
\end{align*}
Then the statement (2) follows from an observation that a change of  variables, $\frac{1}{4s}\mapsto s,$
is a permutation on $\mathbb F_q^*$.
\end{proof}

\subsection{Counting function $\nu(t)$}
We investigate some properties of $\nu(t)$. The $l^2$ estimate of the counting function $\nu$ takes
the following form.
\begin{proposition}\label{Lem3}
For $E,F\subset \mathbb F_q^d$, we have
\begin{align}\label{L1}
\sum_{t\in \mathbb F_q} \nu^2(t)&\leq q^{-1}|E|^2|F|^2 + q^{2d}|F|{\mathfrak M}(E),\mbox{ and }\\
\label{L2} \sum_{t\in \mathbb F_q} \nu^2(t)&\leq q^{-1}|E|^2|F|^2 + q^{3d} \left|\sum_{m\in S_0}
\overline{\widehat{E}}(m) \widehat{F}(m)\right|^2 + q^{2d}|F| {\mathfrak M}^*(E).
\end{align}
\end{proposition}

\begin{proof}
Squaring both sides of (\ref{vt}) and summing over $t\in \mathbb F_q,$ we have
\begin{align}\label{nu:t:sq}
\sum_{t\in \mathbb F_q^d} \nu^2(t)&=q^{4d}\sum_{m,m^\prime\in \mathbb F_q^d} \overline{\widehat{E}}(m)
\widehat{F}(m)
 \widehat{E}(m^\prime) \overline{\widehat{F}}(m^\prime) \sum_{t\in \mathbb F_q}
 \widehat{S_t}(m) \overline{\widehat{S_t}}(m^\prime)
\end{align}
Since $\mathbb F_q^*=\mathbb F_q \setminus \{0\}$, it follows from Proposition \ref{decay} that
(\ref{nu:t:sq}) is equal to
\begin{align}
\frac{|E|^2|F|^2}{q} &+q^{3d-1} \sum_{m,m^\prime\in \mathbb F_q^d} \overline{\widehat{E}}(m)
\widehat{F}(m) \widehat{E}(m^\prime) \overline{\widehat{F}}(m^\prime)\sum_{s\in\mathbb F_q}
\chi(s(\|m\|-\|m^\prime\|))\label{lem3:1}\\\notag &-q^{3d-1} |\sum_{m\in \mathbb F_q^d}
\overline{\widehat{E}}(m) \widehat{F}(m) ~|^2.
\end{align}
Since the last term in (\ref{lem3:1}) is negative, applying the orthogonality relation for $\chi$, we
obtain that
$$\sum_{t\in \mathbb F_q^d} \nu^2(t)\leq \frac{|E|^2|F|^2}{q} +q^{3d}
\sum_{\|m\|=\|m^\prime\|} \overline{\widehat{E}}(m) \widehat{F}(m) \widehat{E}(m^\prime)
\overline{\widehat{F}}(m^\prime),$$ where the last summation is over $m,m'\in \mathbb F_q^d$ with
$\|m\|=\|m'\|$. This can be rewritten as
\begin{equation}\label{good}
\sum_{t\in \mathbb F_q^d} \nu^2(t)\leq \frac{|E|^2|F|^2}{q}+q^{3d} \sum_{r\in \mathbb F_q}
\left|\sum_{m\in S_r} \overline{\widehat{E}}(m) \widehat{F}(m)\right|^2 .
\end{equation}
Observe that
\begin{align}
\sum_{r\in \mathbb F_q} \left|\sum_{m\in S_r} \overline{\widehat{E}}(m) \widehat{F}(m)\right|^2&\leq
\sum_{r\in \mathbb F_q}
\left(\sum_{m\in S_r} |\widehat{E}(m)|^2\right) \left(\sum_{m\in S_r} |\widehat{F}(m)|^2\right)\notag\\
&\leq  {\mathfrak M}(E) \sum_{m\in \mathbb F_q^d} |\widehat{F}(m)|^2= q^{-d} {\mathfrak M}(E)|F|.
\label{lem3:2}
\end{align}
Then (\ref{L1}) follows from this inequality.

Note that (\ref{lem3:2}) is still valid even when ``$r\in\mathbb F_q$" and ${\mathfrak M}(E)$ are
replaced by ``$r\in\mathbb F_q^*$" and ${\mathfrak M}^*(E)$. Since $\mathbb F_q=\mathbb
F_q^*\cup\{0\}$, the inequality (\ref{L2}) is easily derived from (\ref{good}) and the variant of
(\ref{lem3:2}).
\end{proof}

In the next few paragraphs, we collect a number of lemmas on $\nu(0)$, which are going to be utilized
for the proof of Theorem \ref{main2}.

\begin{proposition}\label{evencor}
Suppose that $d\geq 2$ is even and $E, F\subset \mathbb F_q^d$ with $|E||F|\geq 16 q^d.$ Then we have
\begin{align*}
&(|E||F|-\nu(0))^2 \geq \frac{|E|^2|F|^2}{36},\mbox{ and }\\
&q^{3d}  \left|\sum\limits_{m\in S_0}
\overline{\widehat{E}}(m) \widehat{F}(m)\right|^2 -\nu^2(0)\leq q^{-1}|E|^2|F|^2.
\end{align*}
\end{proposition}

\begin{proof}
Since $d\geq 2$ is even, it follows from (\ref{vt}) and (\ref{Sf}) that
\begin{align}
\nu(0)&=q^{2d} \sum_{m\in \mathbb F_q^d}\overline{\widehat{E}}(m) \widehat{F}(m) \left(
q^{-1}\delta_0(m) +q^{-d-1} G^d \sum_{r\in \mathbb
F_q^*}\chi\left(\frac{\|m\|}{4r}\right)\right)\notag\\
&=q^{-1}|E||F| +q^{d-1}G^d \sum_{m\in \mathbb F_q^d} \overline{\widehat{E}}(m) \widehat{F}(m)
\left(\sum_{s\in \mathbb F_q}\chi(s\|m\|) -1\right)\notag\\
&=q^{-1}|E||F| + q^dG^d \sum_{m\in S_0} \overline{\widehat{E}}(m) \widehat{F}(m)-q^{d-1}G^d\sum_{m\in
\mathbb F_q^d} \overline{\widehat{E}}(m) \widehat{F}(m).  \label{v01}
\end{align}
Notice that from (\ref{v01}) we obtain
$$ |\nu(0)|\leq q^{-1}|E||F| +2 q^d |G|^d
\sum_{m\in \mathbb F_q^d} |\overline{\widehat{E}}(m)|~| \widehat{F}(m)|.$$ From the Cauchy-Schwarz
inequality and the Plancherel theorem, we have
\begin{align*}
\sum_{m\in \mathbb F_q^d} |\overline{\widehat{E}}(m)|~| \widehat{F}(m)|&\leq q^{-d} |E|^{\frac{1}{2}}
|F|^{\frac{1}{2}},
\end{align*}
and, therefore
\begin{equation*}
|\nu(0)|\leq q^{-1}|E||F| +2 q^{\frac{d}{2}} |E|^{\frac{1}{2}} |F|^{\frac{1}{2}}.
\end{equation*}
This inequality implies that
$$ |E||F|-\nu(0)\geq |E||F|-|\nu(0)|\geq (1-q^{-1}) |E||F| -2q^{\frac{d}{2}} |E|^{\frac{1}{2}} |F|^{\frac{1}{2}}.$$
Since $q\geq 3$ and $ |E||F|\geq 16 q^d,$ we see that that
$$ |E||F|-\nu(0)\geq \frac{2|E||F|}{3} -2q^{\frac{d}{2}} |E|^{\frac{1}{2}} |F|^{\frac{1}{2}} \geq \frac{|E||F|}{6} \geq 0.$$
The statement follows  immediately from this observation.

For the second statement, let us set $$M(E,F)=q^{-1}|E||F| -q^{d-1}G^d \sum\limits_{m\in \mathbb
F_q^d} \overline{\widehat{E}}(m) \widehat{F}(m)$$ for convenience of calculations. Plugging
(\ref{v01}) into $\nu^2(0)=\nu(0) \overline{\nu(0)}$ and expanding it,
we see that
\begin{align*}
\nu^2(0)=&q^{3d} \left|\sum\limits_{m\in S_0} \overline{\widehat{E}}(m) \widehat{F}(m)\right|^2 + q^d
G^d  \overline{M(E,F)}\sum\limits_{m\in S_0} \overline{\widehat{E}}(m) \widehat{F}(m)\\
&+ q^d \overline{G^d} M(E,F) \sum_{m\in S_0} \widehat{E}(m) \overline{\widehat{F}}(m)+ |M(E,F)|^2.
\end{align*}
Since $\nu^2(0)$ is a nonnegative integer and $|M(E,F)|^2\geq 0,$  the equality above  with
$|G|=q^{\frac{1}{2}}$ implies that
\begin{align*}
q^{3d}  \left|\sum\limits_{m\in S_0} \overline{\widehat{E}}(m) \widehat{F}(m)\right|^2 -\nu^2(0)\leq 2
q^{\frac{3d}{2}} |M(E,F)| \sum_{m\in \mathbb F_q^d} |\widehat{E}(m)| |\widehat{F}(m)|.
\end{align*}
Note that the second factor is  bounded by $q^{-d}|E|^{\frac{1}{2}}|F|^{\frac{1}{2}}$ using the Cauchy-Schwarz
inequality and the Plancherel theorem.
Using this estimate and the definition of $M(E,F)$, it is easy to see that
$$|M(E,F)|\leq q^{-1}|E||F| + q^{d-1}|G^d| \sum\limits_{m\in \mathbb F_q^d} |\widehat{E}(m)| |\widehat{F}(m)| \leq q^{-1}|E||F|+ q^{\frac{d}{2}-1} |E|^{\frac{1}{2}}|F|^{\frac{1}{2}}.$$
Putting  all estimates above together gives that
$$ q^{3d}  \left|\sum\limits_{m\in S_0} \overline{\widehat{E}}(m) \widehat{F}(m)\right|^2 -\nu^2(0) \leq 2q^{\frac{d}{2}-1} |E|^{\frac{3}{2}}|F|^{\frac{3}{2}} +2 q^{d-1} |E||F|.$$
A direct computation shows that if $|E||F|\geq 16 q^d,$ then R.H.S. of previous inequality is less
than or equal to $4 q^{\frac{d}{2}-1} |E|^{\frac{3}{2}}|F|^{\frac{3}{2}} \leq q^{-1} |E|^2|F|^2.$ This
completes the proof.
\end{proof}

\subsection{Spherical sums ${\mathfrak M}(E)$ and ${\mathfrak M}^*(E)$} The following lemma plays a crucial role in proving results in the case of dimension two.
The proof can be found in Chapman {\it et al} \cite[Lemma 4.4]{CEHIK09}.

\begin{lemma}\label{restriction}
If $E\subset \mathbb F_q^2,$ then one has
$$ {\mathfrak M}^*(E) \leq \sqrt{3} q^{-3} |E|^{\frac{3}{2}}.$$
\end{lemma}

For higher dimensions, we need:
\begin{proposition}\label{sph:sum}
For odd $d\geq3$, we have
$${\mathfrak M}(E) \leq \min\{ q^{-d}|E|,~~ 2q^{-d-1}|E|+ 2 q^{-\frac{3d+1}{2}}|E|^2\}.$$
For even $d\geq 2$, we have
$${\mathfrak M}^*(E) \leq \min\{ q^{-d}|E|,~~ 2q^{-d-1}|E|+ 2 q^{-\frac{3d+1}{2}}|E|^2\}.$$
\end{proposition}

\begin{proof}
For each $r\in \mathbb F_q,$ the Plancherel theorem yields $\sum\limits_{m\in S_r} |\widehat{E}(m)|^2 \leq
\sum\limits_{m\in \mathbb F_q^d} |\widehat{E}(m)|^2=q^{-d}|E|.$ Hence we obtain that
\begin{equation}\label{trivial1}
{\mathfrak M}(E)\leq q^{-d}|E|.
\end{equation}
From the definition of the Fourier transform,  it follows that for each $r\in \mathbb F_q,$
\begin{align*}
\sum\limits_{m\in S_r} |\widehat{E}(m)|^2&=q^{-d}\sum_{x,y\in E} \widehat{S_r}(x-y)\\
&=q^{-d} |E| \widehat{S_r}( 0,\dots,0) + q^{-d}\sum_{x,y\in E: x\neq y} \widehat{S_r}(x-y).
\end{align*}
For odd $d\geq 3$, we see  from (\ref{s:0}) and Proposition \ref{decay} (1) that
$${\mathfrak M}(E)\leq 2q^{-d-1}|E|+ 2 q^{-\frac{3d+1}{2}}|E|^2.$$
Combining with (\ref{trivial1}), this estimate yields the first statement.

Finally, similar arguments with the
above give the same upper bound for ${\mathfrak M}^*(E)$.
\end{proof}
When $E$ is of a special form, namely a product set, we have a stronger bound.
\begin{proposition}\label{productpro}
Let $E=\underline{E} \times A \subset \mathbb F_q^{d-1}\times \mathbb F_q.$ Then $\mathfrak M(E)\leq
2q^{-d-1}|A|^2|\underline{E}|$.
\end{proposition}

\begin{proof}
We see from the definition of the Fourier transform that for $m=(\underline m, m_d)\in \mathbb
F_q^{d-1}\times \mathbb F_q,$
$$ \widehat{E}(m) =\widehat{\underline{E}\times A}(\underline{m}, m_d) =\widehat{\underline{E}}(\underline{m}) \widehat{A}(m_d),$$
where $\widehat{\underline{E}}(\underline{m}):=q^{-(d-1)} \sum\limits_{\underline{x}\in \mathbb
F_q^{d-1}} \chi(-\underline{m}\cdot \underline{x})\underline{E}(\underline{x}), $ and
$\widehat{A}(m_d):= q^{-1} \sum\limits_{s\in \mathbb F_q} \chi(-s\cdot m_d) A(s).$ Then, for each
$r\in \mathbb F_q,$ we can write
$$\sum\limits_{m\in S_r} |\widehat{E}(m)|^2 = \sum_{\underline{m} \in \mathbb F_q^{d-1}}
\left|\widehat{\underline{E}}(\underline{m})\right|^2 \left( \sum_{m_d\in \mathbb F_q:
m_d^2=r-\|\underline{m}\|} |\widehat{A}(m_d)|^2\right).  $$ Since $ |\widehat{A}(m_d)|\leq
|\widehat{A}(0)|=q^{-1}|A|$ for all $m_d\in \mathbb F_q,$  and $|\{m_d\in \mathbb F_q:
m_d^2=r-\|\underline{m}\|\}|\leq 2$ for each $r\in \mathbb F_q, \underline{m}\in \mathbb F_q^{d-1},$
we see that
$$\sum\limits_{m\in S_r} |\widehat{E}(m)|^2\leq 2 q^{-2} |A|^2 \sum_{\underline{m}
\in \mathbb F_q^{d-1}}\left|\widehat{\underline{E}}(\underline{m})\right|^2=2q^{-d-1}|A|^2
|\underline{E}|,$$ where the last equality follows from the Plancherel theorem in dimension $(d-1).$
\end{proof}

\section{Distance sets: Main results}\label{main:thms}

Now, we review  standard distance formulas which were originally due to Iosevich and Rudnev
\cite{IR07}.

Basic inequalities the distance function enjoys are as follows.
\begin{lemma}\label{Lem2} Let $E, F\subset \mathbb F_q^d, d\geq 2.$ Then we have
\begin{align}\label{Df1} |\Delta(E,F)|&\geq
\frac{|E|^2|F|^2}{\sum\limits_{t\in \mathbb F_q} \nu^2(t)},\mbox{ and }\\
\label{Df2} |\Delta(E,F)|&\geq \frac{(|E||F|-\nu(0))^2}{\sum\limits_{t\in \mathbb F_q^*} \nu^2(t)}.
\end{align}
\end{lemma}
\begin{proof} Since $|E||F|=\sum_{t\in \mathbb F_q} \nu(t)$ and  $|E||F|-\nu(0)=\sum_{t\in \mathbb F_q^*} \nu(t),$  the Cauchy-Schwarz inequality yields
$$ |E|^2|F|^2=\left(\sum_{t\in \mathbb F_q}\nu(t) \right)^2\leq |\Delta(E,F)| \sum_{t\in \mathbb F_q} \nu^2(t),$$ and
$$ (|E||F|-\nu(0))^2 =\left(\sum_{t\in \mathbb F_q^*}\nu(t) \right)^2\leq |\Delta(E,F)| \sum_{t\in \mathbb F_q^*} \nu^2(t).$$
Thus, (\ref{Df1}) and (\ref{Df2}) follow immediately from these observations.
\end{proof}
\begin{remark} As we shall see, the inequality (\ref{Df1}) is used to prove our distance results in odd dimensions.
On the other hand,  the inequality (\ref{Df2}) is useful in even dimensional case. Iosevich and Rudnev
\cite{IR07} and Dietmann \cite{Di12} made use of the formula (\ref{Df2}) to derive distance results.
Consequently, they obtained the nontrivial distance results in the case when $|E||F|\gg q^d.$ In
this paper we want to point out that if the dimension $d\geq 3$ is odd, then the formula (\ref{Df1})
enables us to yield nontrivial results whenever $|E||F|\gg q^{d-1}.$ In \cite{Di12}, Dietmann
obtained the result in (\ref{Di}) by estimating $\sum_{t\in \mathbb F_q^*} \nu^2(t).$ To the end, he
applied the pigeonhole principle so that his result contains the $\log{q}$ factor. However, our main results below show that the $\log{q}$ factor can be removed.
\end{remark}

\begin{theorem}\label{main1}
If $d\geq 3$ is odd, and $E, F\subset \mathbb F_q^d,$ then
$$ |\Delta(E,F)|\geq \left\{ \begin{array}{ll} \min \left\{ \frac{q}{2}, \frac{|E||F|}{8q^{d-1}}\right\}
\quad &\mbox{if} ~~ 1\leq |E|< q^{\frac{d-1}{2}}\\
\min \left\{ \frac{q}{2}, \frac{|F|}{8q^{\frac{d-1}{2}}}\right\}\quad &\mbox{if} ~~ q^{\frac{d-1}{2}}
\leq |E|< q^{\frac{d+1}{2}}\\
\min \left\{ \frac{q}{2}, \frac{|E||F|}{2q^{d}}\right\}\quad &\mbox{if} ~~q^{\frac{d+1}{2}} \leq
|E|\leq q^{d} \end{array}\right. .
$$
\end{theorem}

\begin{proof}
Combining (\ref{Df1}) with (\ref{L1}),  we see that
\begin{equation}\label{formulaall}
|\Delta(E,F)|\geq \frac{|E|^2|F|^2}{q^{-1}|E|^2|F|^2 + q^{2d}|F| {\mathfrak M}(E)}.
\end{equation}
Since $d$ is odd, from Proposition \ref{sph:sum}, we have
\begin{align*}
{\mathfrak M}(E) \leq \left\{\begin{array}
{ll} 4 q^{-d-1}|E| \quad &\mbox{if}~~ 1\leq |E|< q^{\frac{d-1}{2}}\\
4 q^{-\frac{3d+1}{2}} |E|^2 \quad &\mbox{if}~~q^{\frac{d-1}{2}}\leq |E|< q^{\frac{d+1}{2}}\\
q^{-d} |E|  &\mbox{if}~~q^{\frac{d+1}{2}}\leq |E| \leq q^d. \end{array} \right.
\end{align*}
After combining  (\ref{formulaall}) with this estimate, a direct computation enables us to finish the proof of
Theorem \ref{main1}.
\end{proof}

\begin{remark} It is not hard to see that Theorem \ref{main1} is stronger than Diemann's
result \eqref{Di} for $d\ge 3$ odd (see Corollary \ref{bett:dietm} below).
In addition, notice that  Theorem \ref{main1} improves Shparlinski's result (\ref{Sh}) in the case when
  $d\geq 3$ is odd and $1\leq |E|< q^{(d+1)/{2}}.$
One important point is that  if $|E||F|\leq q^d,$ then  Shparlinski's result
says nothing more than  $|\Delta(E,F)|\geq1.$ The same thing can be said for Dietmann's result,
because his result depends on a strong assumption that $|E||F|\geq (900+\log{q}) q^d.$  In contrast,
Theorem \ref{main1} gives a meaningful information about $|\Delta(E,F)|$ whenever $d\geq 3$ is odd and
$ 8q^{d-1}<|E||F|\leq q^d.$ For example, if $d\geq 3$ is odd, $|E|=q^{(d-1)/2}-1$, and $|F|=
q^{(d+1)/2}$, then $|E||F|=q^d-q^{(d+1)/2} < q^d,$ but $|\Delta(E,F)|\geq q/12.$
\end{remark}

Now, we state and prove our main theorem for even dimensions, which improves Dietmann's result \eqref{Di} (see also Corollary \ref{bett:dietm} below).
\begin{theorem} \label{main2}
(1) Let $d\geq 2$ be even and $E, F\subset \mathbb F_q^d.$ If $|E||F|\geq 16 q^d,$ then we have
$$ |\Delta(E,F)|\geq \left\{
\begin{array}
{ll} \frac{q}{144} \quad&\mbox{for}~~1\le |E|< q^{\frac{d-1}{2}}\\
\frac{1}{144} \min\left\{ q, \frac{|F|}{ 2q^{\frac{d-1}{2}}}\right\} \quad&\mbox{for}~~  q^{\frac{d-1}{2}}\leq |E|< q^{\frac{d+1}{2}}\\
\frac{1}{144}     \min\left\{ q, \frac{2|E||F|}{q^d}\right\} \quad&\mbox{for}~~  q^{\frac{d+1}{2}}\leq
|E|\leq q^{d}.
\end{array}\right.$$ (2) In addition, if $d=2$ and $|E||F|\geq 16 q^2,$ then
$$ |\Delta(E,F)|\geq \frac{1}{72} \min\left\{ \frac{q}{2},~\frac{|E|^{\frac{1}{2}}|F|}{\sqrt{3}q} \right\}. $$
Furthermore, if $d=2$ and $-1\not\in {\mathbb F_q^*}^2$, then
$$ |\Delta(E,F)|\geq \min\left\{ \frac{q}{2}, ~\frac{|E|^{\frac{1}{2}}|F|}{2(\sqrt{3}+1)q} \right\}.$$
\end{theorem}

\begin{proof}
Note that (\ref{L2}) can be rewritten as
\begin{align*}
\sum_{t\in \mathbb F_q^*} \nu^2(t)&\leq q^{-1}|E|^2|F|^2+ q^{2d}|F| {\mathfrak M}^*(E) + q^{3d}
\left|\sum_{m\in S_0} \overline{\widehat{E}}(m) \widehat{F}(m)\right|^2 -\nu(0)^2
\end{align*}
Applying this inequality with ones in Proposition \ref{evencor} to (\ref{Df2}), we obtain that
\begin{equation}\label{shorteven}
|\Delta(E,F)|\geq \frac{|E|^2|F|^2/36}{2q^{-1}|E|^2|F|^2 +  q^{2d}|F| {\mathfrak M}^*(E)}.
\end{equation}
Now, observe that Proposition \ref{sph:sum} implies
\begin{align*}
{\mathfrak M}^*(E) &\leq \left\{\begin{array}{ll} 4 q^{-d-1}|E| \quad &\mbox{if}~~ 1\leq |E|< q^{\frac{d-1}{2}}\\
4 q^{-\frac{3d+1}{2}} |E|^2 \quad &\mbox{if}~~q^{\frac{d-1}{2}}\leq |E|< q^{\frac{d+1}{2}}\\
q^{-d} |E|  &\mbox{if}~~q^{\frac{d+1}{2}}\leq |E| \leq q^d. \end{array} \right.
\end{align*}
Combining this inequality with (\ref{shorteven}) and considering the dominant term in terms of $|E|$,
we obtain from a direct calculation that

$$ |\Delta(E,F)|\geq
\left\{\begin{array}{ll} \frac{1}{144} \min\left\{q, ~\frac{|E||F|}{2 q^{d-1}}\right\}                                             \quad&\mbox{for}~~ 1\leq |E|< q^{\frac{d-1}{2}}\\
\frac{1}{144} \min\left\{ q, ~\frac{|F|}{2q^{\frac{d-1}{2}}}\right \}\quad&\mbox{for}~~ q^{\frac{d-1}{2}} \leq |E|< q^{\frac{d+1}{2}} \\
\frac{1}{144} \min\left \{ q, ~\frac{2|E||F|}{q^d}\right\}\quad&\mbox{for}~~ q^{\frac{d+1}{2}} \leq
|E|\leq q^d. \end{array}\right.$$ Since $|E||F|\geq 16 q^d,$ we conclude the first part of the
theorem.

For the statement (2), suppose that $E, F\subset \mathbb F_q^2$ with $|E||F|\geq 16 q^2.$ Applying
Lemma \ref{restriction} to the inequality (\ref{shorteven}), we conclude that
$$ |\Delta(E,F)|\geq \frac{|E|^2|F|^2}{36(2q^{-1}|E|^2|F|^2 +
\sqrt{3}q|E|^{\frac{3}{2}}|F|)}\geq\frac{1}{72} \min\left\{
\frac{q}{2},~\frac{|E|^{\frac{1}{2}}|F|}{\sqrt{3}q} \right\}$$ and complete the proof.

For the last statement, let $E, F\subset \mathbb F_q^2$. In addition, assume that $-1\in {\mathbb F_q}$
is not a square. Then $S_0=\{(0,0)\}.$ Therefore, it follows from (\ref{L2}) in Lemma \ref{Lem3} that
$$ \sum_{t\in \mathbb F_q} \nu^2(t)\leq q^{-1}|E|^2|F|^2 + q^{6}
\left|\overline{\widehat{E}}(0,0) \widehat{F}(0,0)\right|^2 + q^{4}|F| {\mathfrak
M}^*(E).$$ Since $ \overline{\widehat{E}}(0,0)=q^{-2}|E|$ and
$\widehat{F}(0,0)=q^{-2}|F|,$ an application of  Lemma \ref{restriction} yields that
$$\sum_{t\in \mathbb F_q} \nu^2(t)\leq q^{-1}|E|^2|F|^2 +q^{-2}|E|^2|F|^2 + \sqrt{3}q |E|^{\frac{3}{2}}|F|.$$
Now, observe that $q^{-2}|E|^2|F|^2\leq q |E|^{\frac{3}{2}}|F|$ for $E,F\subset \mathbb F_q^2.$ It
therefore follows that
$$\sum_{t\in \mathbb F_q} \nu^2(t)\leq q^{-1}|E|^2|F|^2 + (1+\sqrt{3})q |E|^{\frac{3}{2}}|F|.$$
By this inequality and (\ref{Df1}) in Lemma \ref{Lem2}, we see that
$$|\Delta(E,F)|\geq \frac{|E|^2|F|^2}{q^{-1}|E|^2|F|^2 + (1+\sqrt{3})q |E|^{\frac{3}{2}}|F|}\geq \min\left\{ \frac{q}{2}, ~\frac{|E|^{\frac{1}{2}}|F|}{2(\sqrt{3}+1)q} \right\}.$$
Thus, the proof is complete.

\end{proof}

Putting Theorem \ref{main1}, \ref{main2} together, we show:
\begin{corollary}\label{bett:dietm}
If $|E||F|\gg q^d$ and $|E|\leq |F|$, then we have
\begin{align*}
|\Delta(E,F)|\gg
\begin{cases}
\min\left\{q,\frac{|F|}{q^{\frac{d-1}{2}}}\right\}&\mbox{ if }d\geq2\\
\min\left\{q,\frac{|E|^{\frac{1}{2}}||F|}{q}\right\}&\mbox{ if }d=2.
\end{cases}
\end{align*}
\end{corollary}

\begin{proof}
First, let $d\geq3$ be odd. Due to the hypothesis, it suffices to consider only the first condition in
Theorem \ref{main1}. In this case, we have $\min\{q/2,|E||F|/8q^{d-1}\}\gg q$. Note also that
$|F|/q^{(d-1)/2}\gg q$ since $|F|\gg q^d/|E|\geq q^{(d+1)/2}$. From similar arguments with Theorem
\ref{main2}, we are also able to verify the corollary for even $d\geq 2$. The case of $d=2$ follows
immediately from Theorem \ref{main2}. Hence the proof is finished.
\end{proof}

\begin{remark}
In dimension two, the second part of Theorem \ref{main2} enables us to improve the first part of
Theorem \ref{main2} in the case when $q\leq|E|\leq q^2.$ As mentioned above, if $d\geq 3$ is odd, then
$|\Delta(E,F)|>1$ whenever $8q^{d-1}<|E||F|\leq q^d.$ However, this is not true any more in even
dimensions as observed in the example (\ref{exa}). For this reason, we need the assumption that
$|E||F|\geq 16 q^d$ in Theorem \ref{main2}. The last part of Theorem \ref{main2} says that if we
assume that $d=2$ and $-1\in \mathbb F_q$ is not a square, then we can drop the assumption in the
second part of Theorem \ref{main2} that $|E||F|\geq 16 q^2.$ In this case, we have  nontrivial
distance results whenever $|E|^{\frac{1}{2}}|F|>2(\sqrt{3}+1)q.$
\end{remark}

The following result can be obtained by finding a good upper bound of ${{\mathfrak M}(E)}$ for any product set $E$ in $\mathbb F_q^d.$
\begin{theorem} \label{main3}
Let $d\geq 2.$ Suppose that $E=A\times A\times \cdots \times A \subset \mathbb F_q^d$ is a product set
and $F\subset \mathbb F_q^d.$ Then we have
\begin{equation}\label{weakproduct}
|\Delta(E,F)|\geq \min \left \{ \frac{q}{2}, \frac{|E|^{1-\frac{1}{d}} |F|}{4q^{d-1}} \right\}.
\end{equation}
\end{theorem}

\begin{proof}
Let $d\geq 2.$ Suppose that $E=A\times A\times \cdots\times A \subset \mathbb F_q^d$ is a product set
and $F\subset \mathbb F_q^d.$ Combining (\ref{Df1}) with (\ref{L1}),  we have
\begin{equation}\label{productform}
|\Delta(E,F)|\geq \frac{|E|^2|F|^2}{ q^{-1}|E|^2|F|^2 + q^{2d}|F| {\mathfrak M}(E)}.
\end{equation}
Let $E=A\times \dots \times A= \underline{E}\times A\subset \mathbb F_q^{d-1}\times \mathbb F_q.$  Then it is clear that $|A|^2
|\underline{E}|= |E|^{1+{1}/{d}}.$ Hence, it follows from Proposition \ref{productpro} that
 $$ {\mathfrak M}(E) \leq 2q^{-d-1}|E|^{1+\frac{1}{d}}.$$
 Applying this inequality to (\ref{productform}) gives
$$|\Delta(E,F)|\geq \frac{|E|^2|F|^2}{ q^{-1}|E|^2|F|^2
+2q^{d-1} |E|^{1+\frac{1}{d}}|F|}\geq \min \left \{ \frac{q}{2}, \frac{|E|^{1-\frac{1}{d}}
|F|}{4q^{d-1}}\right\}.$$ This completes the proof.
\end{proof}

\begin{remark}
Notice that if $q^{d/2}\leq|E|\leq q^d/2^d$, then the conclusion of Theorem \ref{main3} is
superior to the conclusions of  both Theorem \ref{main1} and \ref{main2}. In particular,  the
conclusion of Theorem \ref{main3} holds true without the assumption of Theorem \ref{main2} that $|E||F|\geq 16 q^{d}$.
In fact, it yields nontrivial results whenever $|E|^{1-{1}/{d}}
|F|>4q^{d-1},$  a weaker condition than $|E||F|\geq 16 q^{d}.$
\end{remark}

\begin{remark}
As shown in the proof of Proposition \ref{productpro} and Theorem \ref{main3},  the assumption that $E=A\times A\times \cdots \times A$
can be replaced by a weaker condition that $E=\underline{E}\times A \subset \mathbb F_q^{d-1}\times
\mathbb F_q$ and $|\underline{E}|=|A|^{d-1}.$ In particular, if we assume that $E=\underline{E}\times
A \subset \mathbb F_q^{d-1}\times \mathbb F_q$ and $|\underline{E}|>|A|^{d-1}$, then  we could obtain
much stronger conclusion than  (\ref{weakproduct}).
\end{remark}

\end{document}